\documentclass[3p]{elsarticle}
\usepackage[english]{babel}

\usepackage{subfigure}
\usepackage{amsthm}
\usepackage{amssymb}
\usepackage{amsfonts}
\usepackage{latexsym}
\usepackage{epsfig}
\usepackage{amstext}

\newtheorem{theorem}{Theorem}%[section]
\newtheorem{proposition}{Proposition}%[section]
\newtheorem{corollary}{Corollary}%[section]
\newtheorem{definition}{Definition}%[paper]
\newtheorem{example}{Example}

\begin{document}

\title{On the boundary as an $x$-geodominating set in graphs}

\author[almeria]{J. C\'aceres\fnref{fn1}}
\ead{jcaceres@ual.es}

\author[almeria]{M. Morales\fnref{fn2}}
\ead{mmorale@ual.es}

\author[almeria]{M.L. Puertas\corref{cor1}\fnref{fn1}}
\ead{mpuertas@ual.es}

\address[almeria]{Department of Mathematics, University of Almer\'{\i}a, (Spain)}

\cortext[cor1]{Corresponding author}
\fntext[fn1]{Partially supported by J. Andaluc\'{\i}a FQM-305, EUROCORES programme EuroGIGA - ComPoSe IP04.}
\fntext[fn2]{Partially supported by J. Andaluc\'{\i}a FQM-305.}

\begin{abstract}
Given a graph $G$ and a vertex $x\in V(G)$, a vertex set $S \subseteq V(G)$ is an $x$-geodominating set of $G$ if each vertex $v\in V(G)$ lies on an $x-y$ geodesic for some element $y\in  S$. The minimum cardinality of an $x$-geodominating set of $G$ is defined as the $x$-geodomination number of $G$, $g_x(G)$, and an $x$-geodominating set of cardinality $g_x(G)$ is called a $g_x$-set and it is known that it is unique for each vertex $x$. We prove that, in any graph $G$, the $g_x$-set associated to a vertex $x$ is the set of boundary vertices of $x$, that is $\partial(x)= \{ v \in V(G) : \forall w \in N(v): d(x,w) \leq d(u, v)\}$. This characterization of $g_x$-sets allows to deduce, on a easy way, different properties of these sets and also to compute both $g_x$-sets and $x$-geodomination number $g_x(G)$, in graphs obtained using different graphs products: cartesian, strong and lexicographic.
\end{abstract}


\maketitle

\section{Introduction}

Although convexity notions in graph theory have been traditionally studied from an abstract point of view, very recently one can found in the literature a growing interest in computational aspects of the field. Particularly, the geodetic closure is analyzed in~\cite{CMOP,DPRS} as a tool for reconstructing the entire graph or a certain vertex subset from only a few points in it, similarly to the Euclidean case. Recall that, given a graph $G$ and the interval $I[u,v]$ between two vertices $u,v\in V(G)$ is the set of vertices lying in some shortest path between $u$ and $v$ and for a vertex set $S\subseteq V(G)$, the {geodetic closure} of $S$ is $I(S)=\bigcup_{u,v\in S} I[u,v]$ (see \cite{CHZ}). Also a vertex set $S$ is called geodetic if $I(S)=V(G)$.

One of the obstacles to compute $I(S)$ may be that the worst-case complexity of finding all the shortest paths among all the vertices in $S$ is $O(n|S|)$. However, it is well-known that for some particular cases it is not necessary to explore all the pairs of vertices. In the case of trees, for instance, it is enough to consider shortest paths between one fixed vertex in the geodetic set and the rest of vertices in it. So the computational cost is lower than in the general case. This idea is collected in \cite{ST1}, and developed in \cite{ST2}, where authors go over the path between a fixed vertex and the rest of the vertices in the graph.

\begin{definition} \cite{ST1}
A vertex $y$ in a connected graph $G$ is said to $x$-geodominates a vertex $u$ if $u$ lies on an $x-y$ geodesic. A set of vertices of $G$ is an $x$-geodominating set if each vertex $v\in V(G)$ is $x$-geodominated by some element of $S$. The minimum cardinality of an $x$-geodominating set of $G$ is defined as the $x$-geodomination number of $G$, denoted by $g_x(G)$ or $g_x$. An $x$-geodominating set of cardinality $g_x(G)$ is called a $g_x$-set.
\end{definition}

In \cite{ST1} authors prove that there exists an unique $g_x$-set, for each vertex $x$ of a graph $G$. This property leads to the obvious question of describe it. Our goal is to find a description using local distance properties, so it can be computed in an easy way. As a first idea we have in mind the behavior of leaves of trees. In order to rebuild the tree using the geodetic interval operation, it is enough to consider shortest paths between one fixed vertex and all the leaves. So leaves of a tree are a $x$-geodominating set, for any vertex $x$ in the tree. In \cite{ST1} it is shown that every simplicial vertex of $G$ belongs to the $g_x$-set for any vertex $x$ in $G$, so it is natural to wonder if the simplicial vertex set is an $x$-geodominating set. Of course the answer is affirmative for trees, but it is negative, for instance, in interval graphs, which is a family of perfect graphs with a simple structure. The graph in Figure~\ref{extremos} is an interval graph with simplicial vertex set $\{ x,u,v,w\}$, which it is not a $z$-geodominating set, for any vertex $z$ in the graph.

\begin{figure}[ht]
  \begin{centering}
    \includegraphics[width=0.25\textwidth]{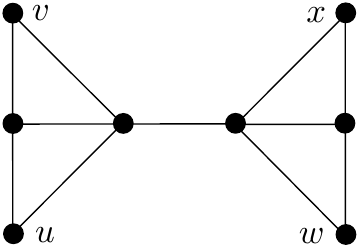}

  \caption{Simplicial vertices are not an $x$-geodominating set.}
  \label{extremos}
  \end{centering}
\end{figure}

We turn our attention to a different set. Given a graph $G$, the boundary of a vertex $x$ in $G$ is the set $\partial(x)= \{ v \in V(G) : \forall w \in N(v): d(x,w) \leq d(u, v)\}$ (see \cite{CEJZ}).
In~\cite{HMPS} it is shown that the boundary of a vertex, union the vertex itself, is a geodetic set so following this idea we prove our main result.

\section{Main result}

In this Section we prove that, in any graph $G$, the unique $g_x$-set associated to a vertex $x$ is the set $\partial(x)$ of boundary vertices of $x$.

\begin{proposition}\label{proposition}
Let $x$ be a vertex of a graph $G$. Then $\partial (x)$, the boundary of $x$, is an $x$-geodominating set.
\end{proposition}

\begin{proof}
Let $x$ be a vertex in a graph $G$ and consider $u\in V(G)\setminus \partial(x)$. By definition, there exist $u_1\in N(u)$ such that $d(x,u)<d(x,u_1)=d(x,u)+1$, so $u\in I[x,u_1]$. If $u_1\in \partial(x)$, we are done. If $u_1\notin \partial(x)$, there exist $u_2\in N(u_1)$ such that $d(x,u_1)<d(x,u_2)=d(x,u_1)+1$, so $I[x,u_1]\subseteq I[x,u_2]$ and $u\in I[x,u_2]$. If $u_2\in \partial(x)$, we are done and if $u_2\notin \partial(x)$, we can repeat this process, as many times as necessary, and finally find a vertex $u_k\in \partial(x)$ such that $u\in I[x,u_k]$.
\end{proof}

We can now state our main theorem.

\begin{theorem}\label{teorema}
Let $x$ be a vertex of a graph $G$. Then a set $S$ of vertices is $x$-geodominating if and only if $\partial(x)\subseteq S$.
\end{theorem}

\begin{proof}
Suppose that $S$ is an $x$-geodominating set and assume that $\partial(x)\setminus S\neq \emptyset$. Then there exists $v\in \partial(x)$ such that $v\notin S$. By hypothesis $v\in I[x,u]$, for some $u\in S$ and using that $u\neq v$, there is a neighbor $w$ of $v$ in the shortest path between $x$ and $u$ such that $d(x,v)<d(x,w)=d(x,v)+1$, a contradiction with $v\in \partial(x)$. Conversely, if $\partial(x)\subseteq S$, it is clear that $\bigcup_{v\in \partial(x)} I[x,v] \subseteq \bigcup_{s\in S} I[x,s]$, so using Proposition~\ref{proposition}, $V(G)=\bigcup_{s\in S} I[x,s]$ as desired.
\end{proof}

Finally the following Corollary, that describes $g_x$-sets, can be immediately deduced from the Theorem above.

\begin{corollary}\label{corolario}
Let $x$ be a vertex of a graph $G$, then the unique $g_x$-set is $\partial(x)$ and the $x$-geodomination number is $g_x=|\partial(x)|$.
\end{corollary}

This description of $g_x$-sets as the boundary of vertex $x$ reveals the truly nature of this boundary set, in terms of convexity properties: it is the set needed to rebuild the graph by means the interval operation, with one fixed end. Also note that several results in \cite{ST1} describing basic properties of $g_x$-sets, for example Theorems 2.4, 2.5, 2.12 and 2.14, can be easily deduced from Corollary~\ref{corolario}.

\section{Applications: $x$-geodomination in graph products}

We recall different product operations in graphs (see \cite{IK}). We denote by $E(G)$ the set of edges of a graph $G$. Let $G$ and $H$ be two graphs, then a graph with vertex set $V(G)\times V(H)$ is called:
\begin{itemize}
\item the cartesian product of $G$ and $H$, denoted by $G\Box H$ if: $(g,h),(g',h')\in E(G\Box H)$ if and only if $gg'\in E(G), h=h'$ or $g=g', hh'\in E(H)$,
\item the lexicographic product of $G$ and $H$, denoted by  $G\circ H$ if: $(g,h),(g',h')\in E(G\circ H)$ if and only if $gg'\in E(G)$ or $g=g', hh'\in E(H)$,
\item the strong product of $G$ and $H$, denoted by $G\boxtimes H$ if: $(g,h),(g',h')\in E(G\boxtimes H)$ if and only if $gg'\in E(G), h=h'$ or $g=g', hh'\in E(H)$  or
$gg'\in E(G), hh'\in E(H)$.
\end{itemize}

In order to compute the $x$-geodominating number in a product graph, in terms of geodominating numbers on factors, it will be enough to describe the boundary of a vertex in the product graphs in terms of the boundary of factors.

\begin{proposition}\label{products}
Let $G$ and $H$ be graphs and let $g\in V(G)$ and $h\in V(H)$ be. Then:
\begin{enumerate}
\item $\partial_{G\Box H} \big((g,h)\big)=\partial_{G} (g)\times \partial_{H} (h)$.
\item $\{ g\} \times \partial_{H} (h)\subseteq \partial_{G\circ H} \big((g,h)\big)\subseteq \big(\partial_{G} (g)\times V(H)\big) \bigcup \big( \{ g\} \times \partial_{H} (h)\big)$.
\item $\partial_G(g)\times \partial_H(h)\subseteq \partial_{G\boxtimes H} \big((g,h)\big)\subseteq \big(\partial_{G} (g)\times V(H)\big) \bigcup \big(V(G)\times \partial_{H} (h)\big)$.
\end{enumerate}

\end{proposition}

\begin{proof}
\begin{enumerate}
\item Let $(g',h')\in \partial_{G\Box H} \big((g,h)\big)$ be. We will see that $g'\in \partial_{G} (g)$ (and analogously $h'\in \partial_{H} (h)$). For each $g''\in N(g')$, it is clear that $(g'',h')\in N\big((g',h')\big)$, so $d\big((g,h),(g',h')\big)\geq d\big((g,h),(g'',h')\big)$ and using the distance properties of the cartesian product, $d(g,g')+d(h,h')\geq d(g,g'')+d(h,h')$, that implies $d(g,g')\geq d(g,g'')$, as desired.

    Conversely let $(g',h')\in \partial_{G} (g)\times \partial_{H} (h)$ and $(g'',h'')\in N\big((g',h')\big)$ be. Suppose, without lost of generality that $g'=g''$ and $h''\in N(h')$. Then, by hypothesis, $d\big((g,h),(g',h')\big)=d(g,g')+d(h,h')=d(g,g'')+d(h,h')\geq d(g,g'')+d(h,h'')=d\big((g,h),(g'',h'')\big)$, as desired.

\item Let $(g,h')\in \{ g\} \times \partial_{H} (h)$ and $(g'',h'')\in N\big((g,h')\big)$ be. If $g''\in N(g)$, using the distance properties in the lexicographic product, we obtain $d\big((g,h),(g'',h'')\big)=d(g,g'')=1\leq d(h,h')=d\big((g,h),(g,h')\big)$. On the other hand, if $g''=g$ and $h''\in N(h')$ then, by hypothesis, $d\big((g,h),(g'',h'')\big)=d(h,h'')\leq d(h,h')=d\big((g,h),(g,h')\big)$. Thus $(g,h')\in \partial_{G\circ H} \big((g,h)\big)$.

    To prove the second inclusion, let $(g',h')\in \partial_{G\circ H} \big((g,h)\big)$ be. We consider two cases: $g'\neq g$ and $g'=g$. Firstly, if $g'\neq g$, let see that $g'\in \partial_{G}(g)$. Let $g''\in N(g')$ be, then $(g'',h')\in N\big((g',h')\big)$ and, by hypothesis, $d\big((g,h),(g'',h')\big)\leq d\big((g,h),(g',h')\big)=d(g,g')$, so $d(g,g'')\leq d\big((g,h),(g'',h')\big)\leq d(g,g')$ and $g'\in \partial_{G}(g)$, as desired. Finally if $g'=g$, let see that $h'\in \partial_{H}(h)$. Let $h''\in N(h')$ be, then $(g,h'')\in N\big((g,h')\big)$ and, by hypothesis, $d\big((g,h),(g,h'')\big)\leq d\big((g,h),(g,h')\big)$, so $d(h,h'')\leq d(h,h')$ and $h'\in \partial_{H}(h)$, as desired.

\item  Let $(g',h')\in \partial_G(g)\times \partial_H(h)$ be and let see that $(g',h')\in \partial_{G\boxtimes H} \big((g,h)\big)$. To this end consider $(g'',h'')\in N\big((g',h')\big)$, and suppose that $g''\in N(g'), h'=h''$. Using the distance properties of the strong product of graphs and the hypothesis, we obtain $d\big((g,h),(g',h')\big)=\max \{ d(g,g'),d(h,h')\}=\max \{ d(g,g'),d(h,h'')\}\geq \max \{ d(g,g''),d(h,h'')\}=d\big((g,h),(g'',h'')\big)$, as desired. The reasoning is similar if $g'=g'', h''\in N(h')$. Now suppose that $g''\in N(g'), h''\in N(h')$, then $d\big((g,h),(g',h')\big)=\max \{ d(g,g'),d(h,h')\}\geq \max \{ d(g,g''),d(h,h'')\}=d\big((g,h),(g'',h'')\big)$.

    Finally let $(g',h')\in \partial_{G\boxtimes H} \big((g,h)\big)$ be and let see that either $g'\in \partial_{G} (g)$ or $h'\in \partial_{H} (h)$. Suppose that $g'\notin \partial_{G} (g)$, and let $h''\in N(h')$ be. Then there exists $g''\in N(g')$ such that $d(g,g')<d(g,g'')$ and we consider the vertex $(g'',h'')$ that belongs to $N\big((g',h')\big)$. By hypothesis, $d\big((g,h),(g',h')\big)\geq d\big((g,h),(g'',h'')\big)$, so $\max \{ d(g,g'),d(h,h')\}\geq \max \{ d(g,g''),d(h,h'')\}$. Note that condition $d(g,g')<d(g,g'')$ implies that $d(h,h')=\max \{ d(g,g'),d(h,h')\}\geq \max \{ d(g,g''),d(h,h'')\}\geq d(h,h'')$ and $h'\in \partial_{H} (h)$ as desired.
\end{enumerate}\end{proof}

The following Examples shows different situations respect to both lexicographic product and strong product.

\begin{example}

The graph in Figure~\ref{lexicographic}(a) shows the lexicographic product graph $P_3\circ P_3$. Note that  $\partial_{P_3\circ P_3}\big((a,1)\big)=\{ (c,1),(c,2),(c,3),(a,3)\}= \big(\partial_{P_3}(a)\times V(P_3)\big) \bigcup \big(\{ a\} \times \partial_{P_3}(1)\big)$.

However in the same graph (see Figure~\ref{lexicographic}(b)) we have $\partial_{P_3\circ P_3}\big((b,1)\big)=\{(b,3)\} = \{b\}\times \partial_{P_3}(1)$.

\begin{figure}[ht]
  \begin{centering}

    \subfigure[Vertex $(a,1)$ and its boundary.]{\includegraphics[width=0.32\textwidth]{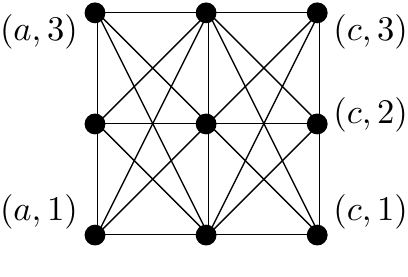}} \hspace{1.5cm}
    \subfigure[Vertex $(b,1)$ and its boundary.]{\includegraphics[width=0.19\textwidth]{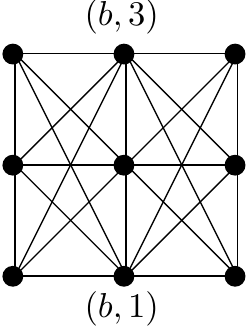}}

    \caption{Different situations happen in the lexicographic product $P_3\circ P_3$.}
  \label{lexicographic}
  \end{centering}
\end{figure}

\end{example}

\begin{example}

The strong product graph $P_3\boxtimes P_3$ is shown in Figure~\ref{strong}(a). Note that $\partial_{P_3\boxtimes P_3}\big((a,1)\big)=\{ (a,3),(b,3),(c,1),(c,2),(c,3)\} =\big(\partial_{P_3}(a)\times V(P_3)\big) \bigcup \big(V(P_3)\times \partial_{P_3}(1)\big)$.

However the graph in Figure~\ref{strong}(b) shows the strong product graph $P_3\boxtimes P_4$ and in this case $\partial_{P_3\boxtimes P_4}\big((a,1)\big)=\{ (a,4),(b,4),(c,1),(c,2),(c,4)\} \varsubsetneqq\big(\partial_{P_3}(a)\times V(P_4)\big) \bigcup \big(V(P_3)\times \partial_{P_4}(1)\big)=\{ (a,4),(b,4),(c,1),(c,2),(c,3),(c,4)\}$.
\begin{figure}[ht]
  \begin{centering}
    \subfigure[Vertex $(a,1)$ and its \newline boundary in $P_3\boxtimes P_3$.]{\includegraphics[width=0.25\textwidth]{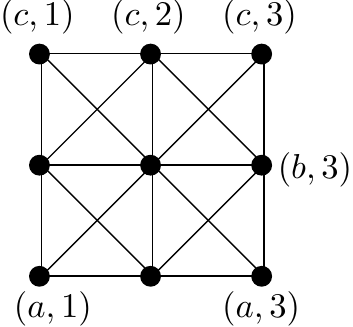}}
    \hspace{2cm}
    \subfigure[Vertex $(a,1)$ and its boundary in \newline $P_3\boxtimes P_4$.]{\includegraphics[width=0.32\textwidth]{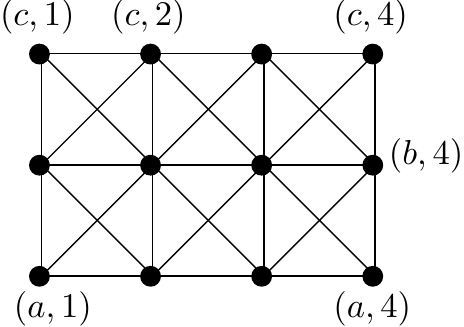}}

   \caption{Different situations happen in the strong product of graphs.}
  \label{strong}
  \end{centering}
\end{figure}

 \end{example}

The following result, about the $x$-geodomination number in graphs products, can be immediately deduced from Proposition~\ref{products}.

\begin{proposition}
Let $G$ and $H$ be graphs and let $x\in V(G)$ and $y\in V(H)$ be. Then:
\begin{enumerate}
\item $g_{(x,y)}(G\Box H)=g_x(G)g_y(H)$.
\item $g_y(H)\leq g_{(x,y)}(G\circ H)\leq g_x(G)|V(H)| + g_y(H)$.
\item $g_x(G) g_y(H) \leq g_{(x,y)}(G\boxtimes H)\leq g_x(G)|V(H)| + |V(G)|g_y(H)$.

\end{enumerate}

\end{proposition}

\section{Conclusions}

In this note we describe the minimum $x$-geodominating set, for any vertex $x$ of a graph $G$, in terms of the boundary condition. With this description, it is trivial to deduce some properties of $g_x$-sets and parameter $g_x(G)$ and also it is easy to calculate exact values for different graphs. Our main theorem describes the truly essence of the boundary vertices as a geodetic set, showing the exact property that allows to reconstruct the graph by means of the interval operation, using boundary vertices. Finally we would like to point out the computational side of the problem. The calculation of the geodetic number is an NP-hard problem for general graphs (see \cite{A}), and the parameter $g_x(G)$ is an alternative, polynomially computable. The obvious relationship between the geodetic number of a graphs $g(G)$ and $g_x(G)$, gives that $g(G)\leq \min \{g_x(G)\colon x\in V(G)\} +1$. Finally, having in mind the boundary condition, the cardinal of the smallest $g_x$-set of a graph $G$ can be computed in time $\Theta(n^2)$. This provides, in polynomial time, a geodetic set (not minimum in most cases), but that rebuilds the graph with an interval operation with lower computational cost that the general one.

\end{document}